\documentclass[10pt,a4paper,reqno]{amsart}

\usepackage{amssymb}
\usepackage{enumitem}

\usepackage[ruled,titlenumbered,noend]{algorithm2e}

\SetAlgoNoLine
\DontPrintSemicolon
\SetNlSty{footnotesize}{}{}
\IncMargin{0.5em}
\SetNlSkip{1em}
\SetKwIF{If}{ElseIf}{Else}{if}{then}{else if}{else}{endif}
\SetNlSkip{2em}

\newenvironment{algorithm-hbox}{\hbadness=10000\begin{algorithm}}{\end{algorithm}}

\usepackage[
    pdftex, colorlinks=true, linkcolor=black, urlcolor=black, citecolor=black,
    pdfpagelabels, 
bookmarks=false,
		pdftitle={Nonrepetitive sequences}, pdfdisplaydoctitle,
]{hyperref}

\theoremstyle{plain}

\newtheorem*{claim}{Claim}

\newtheorem{conjecture}{Conjecture}

\newtheorem{theorem}{Theorem}

\newtheorem*{obs}{Observation}

\numberwithin{equation}{section}
\newcommand{\itemref}[1]{\ref{#1}}

\newenvironment{enumeratei}{\begin{enumerate}[label=\textup{(\roman*)}, noitemsep, topsep=1.5mm, labelindent=.8em, leftmargin=*, widest=.]}{\end{enumerate}}
\newenvironment{enumeratenum}{\begin{enumerate}[label=\textup{(\arabic*)}, noitemsep, topsep=1.5mm, labelindent=.8em, leftmargin=*, widest=.]}{\end{enumerate}}
\renewenvironment{itemize}{\begin{enumerate}[label=$*$, noitemsep, topsep=1.5mm, labelindent=.8em, leftmargin=*, widest=.]}{\end{enumerate}}

\newcommand{\set}[1]{\{#1\}}
\newcommand{\norm}[1]{{|#1|}}

\newcommand{\Cpl}{\mathbb{C}}

\renewcommand{\leq}{\leqslant}
\renewcommand{\geq}{\geqslant}

\DeclareMathOperator{\type}{type}

\begin{document}
\title[A new approach to nonrepetitive sequences]{A new approach to nonrepetitive sequences}
\author[J. Grytczuk]{Jaros\l aw Grytczuk}\thanks{Research of J.\ Grytczuk was supported by the Polish Ministry of Science and Higher Education grant N206257035. Research of J.\ Kozik and P.\ Micek was supported by the National Funds Center within project no.\  DEC-2011/01/D/ST1/04412.}
\address{Theoretical Computer Science Department, Jagiellonian University\\
Krak\'{o}w, Poland \and
Faculty of Mathematics and Information Science, Warsaw University of Technology\\
Warszawa, Poland}
\email{grytczuk@tcs.uj.edu.pl}
\urladdr{http://tcs.uj.edu.pl/Grytczuk}
\author[J. Kozik]{Jakub Kozik}
\author[P. Micek]{Piotr Micek}
\address{Theoretical Computer Science Department, Jagiellonian University\\
Krak\'{o}w, Poland}
\email{jkozik@tcs.uj.edu.pl}
\urladdr{http://tcs.uj.edu.pl/KozikJ}
\email{piotr.micek@tcs.uj.edu.pl}
\urladdr{http://tcs.uj.edu.pl/Micek}
\date{\today}
\keywords{Thue, nonrepetitive sequence}

\begin{abstract}
A sequence is \emph{nonrepetitive} if it does not contain two adjacent identical blocks. The remarkable construction of Thue asserts that $3$ symbols are enough to build an arbitrarily long nonrepetitive sequence. It is still not settled whether the following extension holds: for every sequence of 3-element sets $L_1,\ldots, L_n$ there exists a nonrepetitive sequence $s_1, \ldots, s_n$ with $s_i\in L_i$. We propose a new non-constructive way to build long nonrepetitive sequences and provide an elementary proof that sets of size $4$ suffice confirming the best known bound. The simple double counting in the heart of the argument is inspired by the recent algorithmic proof of the Lov\'{a}sz local lemma due to Moser and Tardos. Furthermore we apply this approach and present game-theoretic type results on nonrepetitive sequences. \emph{Nonrepetitive game} is played by two players who pick, one by one, consecutive terms of a sequence over a given set of symbols. The first player tries to avoid repetitions, while the second player, in contrast, wants to create them. Of course, by simple imitation, the second player can force lots of repetitions of size $1$. However, as proved by Pegden, there is a strategy for the first player to build an arbitrarily long sequence over $37$ symbols with no repetitions of size greater than $1$. Our techniques allow to reduce $37$ to $6$. Another game we consider is the \emph{erase-repetition game}. Here, whenever a repetition occurs, the repeated block is immediately erased and the next player to move continues the play. We prove that there is a strategy for the first player to build an arbitrarily long nonrepetitive sequence over $8$ symbols.
\end{abstract}

\maketitle

\section{Introduction}

A \emph{repetition} of size $h$ in a sequence $S$ is a
subsequence of consecutive terms of $S$ consisting of two identical blocks 
$x_{1}\ldots x_{h}x_{1}\ldots x_{h}$. A sequence is \emph{nonrepetitive}
if it does not contain a repetition of any size $h\geqslant 1$. For
instance, the sequence $1232312$ contains a repetition $2323$ of size two,
while $123132123$ is nonrepetitive.

It is easy to see that each binary sequence of length at least four contains
a repetition. In 1906 Thue \cite{Thu06} proved that $3$
symbols are sufficient to produce arbitrarily long nonrepetitive
sequences (see \cite{Ber95}). His method is constructive and uses
substitutions over a given set of symbols. For instance, the substitution%
\begin{align*}
1& \rightarrow 12312 \\
2& \rightarrow 131232 \\
3& \rightarrow 1323132
\end{align*}%
preserves the property of nonrepetitiveness on the set of finite sequences
over $\{1,2,3\}$. This means that replacing all symbols in a nonrepetitive
sequence by the assigned blocks results in a sequence that still does not
contain repetitions. Sequences generated by substitutions have found many
unexpected applications in such diverse areas as group theory, universal
algebra, number theory, ergodic theory, and formal language theory. The work
of Thue inspired a stream of research leading to emergence of new branches
of mathematics with a variety of challenging open problems (see \cite{AS03,BEM79,Cur05,Gry08,Lot83}).

In this paper we present a different approach to creating long nonrepetitive
sequences. Consider the following naive procedure: generate consecutive
terms of a sequence by choosing symbols at random (uniformly and
independently) and every time a repetition occurs, erase the repeated block
and continue. For instance, if the generated sequence is $12323$, we must
cancel the last two symbols, which brings us back to $123$.

We prove by a simple counting that with positive probability the
length of a constructed sequence exceeds any finite bound, provided the
number of symbols is at least $4$. This is slightly weaker than Thue's
result, but our argument remains valid in more general settings, in which
the method of substitutions does not seem to work.

One particular example of such a setting is the list-version of
nonrepetitive sequences -- an analog of the classical graph choosability
introduced by Vizing \cite{Viz76} and independently by Erd\H{o}s, Rubin,
and Taylor \cite{ERT80}. Suppose we are given a collection of \emph{lists}
(sets of symbols) $L_{1},\ldots ,L_{n}$. A sequence $s_{1}\ldots s_{n}$ is 
\emph{chosen from lists} $L_{1},\ldots ,L_{n}$ if $s_{i}\in L_{i}$ for all $%
i=1,\ldots ,n$. The following list-version of Thue's theorem seems plausible.
\begin{conjecture}
\label{conj}
For every $n\geqslant 1$ and a sequence of sets $L_{1},\ldots ,L_{n}$, each of size 3, there is a nonrepetitive sequence chosen from $L_{1},\ldots ,L_{n}$.
\end{conjecture}

Notice that the statement of the conjecture is not obvious, even for lists of any given size. 
However, a rather straightforward touch of the Lov\'{a}sz local lemma assures that the conjecture is true for sufficiently large lists (for a careful introduction to the local lemma and the probabilistic method in general we send the reader to \cite{AS08}). 
In fact, the bound $64$ comes as a special case of a result on nonrepetitive colorings of
bounded degree graphs (Alon et al. \cite{AGHR02}; see also \cite{Gry07}). Recently Grytczuk, Przyby\l o and Zhu \cite{GPZ} proved that lists of sizes at least $4$ suffice. They achieve this almost tight bound applying an enhanced version of the local lemma due to Pegden \cite{Peg11}. In Section \ref{sec:list-Thue} we give a simple argument for the same bound.

This research would not emerge without a contribution of Moser on his way to an algorithmic proof of Lov\'{a}sz local lemma \cite{MT10}: his entropy compression argument. This was widely discussed in the combinatorics community and we send the reader to great expositions of the topic by Tao \cite{TaoBlog} and Fortnow \cite{ForBlog}.

In this paper we make use of the above-mentioned approach to games involving nonrepetitve sequences.

The \emph{nonrepetitive game} over a symbol set $S$ is played by two players in the following way. The players collectively build a sequence choosing from $S$, one by one, consecutive terms of the sequence. The first player, Ann, is trying to avoid repetitions, while the second player, Ben, does not necessarily cooperate. Of course, just by mimicking Ann's moves Ben can force a lot of repetitions of size $1$. It turns out however that for large enough $S$ he cannot force any larger repetition at all! Pegden \cite{Peg11}, using his extension of the Lov\'{a}sz local lemma, proved that Ann has a strategy in the nonrepetitive game to build an arbitrarily long sequence with no repetition of size greater than $1$ over symbol set of size at least $37$ (no matter how perfidiously Ben is playing). In this paper we prove (Theorem \ref{thm:Pegdens-game}) that Ann can do the same on every set of symbols of size at least $6$. On the other hand, Ben can easily force nontrivial repetitions in a game on just $3$ symbols (see \cite{Peg11}). Thus, the minimum size of a set of symbols required to ensure Ann's strategy is $4$, $5$ or $6$.

The \emph{erase-repetition game} over a set of symbols $S$ is also a two-player game between Ann and Ben. As before they build a sequence picking symbols alternately from $S$ and appending them to the end of the sequence built so far. But this time whenever a repetition occurs the second instance of the repeated block is immediately erased and the next player continues extending the remaining prefix of the sequence. We prove (Theorem \ref{thm:erase-repetition-game}) that there is a strategy for Ann in this game to build an arbitrarily long nonrepetitive sequence over at least $8$ symbols.

The paper is organized as follows. Section \ref{sec:list-Thue} contains the generic argument proving that from any sequence of lists, each of size $4$, one can choose a nonrepetitive sequence. Section \ref{sec:preliminaries} introduce a bit of generating functions theory used in counting arguments. Sections \ref{sec:erase-repetition-game} and \ref{sec:nonrepetitive-game} are devoted to erase repetiton game and nonrepetitive game, respectively.


\section{The algorithm}\label{sec:list-Thue}
Consider the following randomized algorithm. The input is a sequence of lists $L_1, \ldots, L_n$. Random elements are chosen independently with uniform distribution.
\begin{algorithm-hbox}[!ht]
    \caption{Choosing a nonrepetitive sequence from lists of size $4$} 
    \label{alg-choosing-from-lists}
    $i \gets 1$\;
    \While{$i\leq n$}{
      $s_i\gets$ random element of $L_i$\;
      \uIf{\textup{$s_1,\ldots,s_i$ is nonrepetitive}}{
      $i\gets i+1$\;
      }
      \uElse{
	there is exactly one repetition, say $s_{i-2h+1},\ldots,s_{i-h},s_{i-h+1},\ldots,s_i$\;
	$i\gets i-h+1$
      }
    }
\end{algorithm-hbox}


The general idea is that if Algorithm \ref{alg-choosing-from-lists} works long enough for all evaluations of the random experiments, then a lot of repetitions occur, based on which we can compress a random string to a better extent than is actually possible.

\begin{theorem}\label{thm:4}
For every $n\geqslant 1$ and a sequence of sets $L_{1},\ldots ,L_{n}$, each of size 4, there is a nonrepetitive sequence chosen from $L_{1},\ldots ,L_{n}$.
\end{theorem}
\begin{proof}
Suppose for a contradiction that it is not possible to choose a nonrepetitive sequence from $L_1,\ldots,L_n$. This means that Algorithm 1 does not terminate on this sequence. In the following, by the $j$-th step of the algorithm, we mean the $j$-th iteration of the \texttt{while} loop. Set $M$ to be a sufficiently large integer. We are going to record, in two different ways, the possible scenarios of what algorithm does in the first $M$ steps.

Order arbitrarily the elements of each $L_i$. In each step the algorithm picks a random element from a list of size $4$. Let $r_j$ ($1\leq j\leq M$) be the position of the chosen element in the appropriate list. Clearly, $r_1,\ldots,r_M$ is a sequence of random variables with $4^M$ possible evaluations. When we fix evaluations of $r_1,\ldots,r_M$ we make Algorithm~\ref{alg-choosing-from-lists} deterministic.

For fixed evaluations of $r_1, \ldots, r_n$, let $d_1=1$ and $d_j$ ($2\leq j\leq M$) be the difference between the values of variable $i$ after $j$th and ($j-1$)th steps of the algorithm. The important properties are:
\begin{enumeratei}
 \item $d_j\leq 1$, for all $1\leq j\leq M$,\label{item-dj<=1}
 \item $\sum_{j=1}^k d_j\geq 1$, for all $1\leq k\leq M$.\label{item-sum-of-the-prefix>=1}
\end{enumeratei}
A pair $(D,S)$ is a \emph{log} if there is an evaluation of $(r_1,\ldots,r_M)$ such that $D$ is the corresponding sequence of differences and $S$ is the final sequence produced after $M$ steps of the algorithm. The key point is that a log encodes all values of $r_1,\ldots,r_M$ in a lossless fashion.
\begin{claim}
Every log corresponds to a unique evaluation of $r_1,\ldots,r_M$.
\end{claim}
\begin{proof}[Proof of Claim]
Given a log $((d_1,\ldots,d_M),S_M)$ with $S_M=(s_1,\ldots,s_l)$ we are going to decode the evaluation of $r_M$ (the last random choice taken) and $S_{M-1}$ -- the sequence constructed after $M-1$ steps. Then by simple iteration one can extract all the remaining values of $r_{M-1},\ldots,r_1$.

If $d_M=1$ then, the element generated in the $M$th step is appended to the end of $S_M$. Thus the value of $r_M$ is simply the position of $s_l$ in $L_l$. Moreover, no repetition occurred after the $M$th step and therefore $S_{M-1}=(s_1,\ldots,s_{l-1})$.

If $d_{M}\leq0$ then some symbols were erased after the $M$th step. But, since we know the size of the repeated sequence, namely $h=\norm{d_M}+1$, and only one part of it was erased, we can read and copy the appropriate block to restore the sequence before the erasure $(s_1,\ldots,s_{l},s_{l-h+1},\ldots,s_{l})$. Then we read the value of $r_M$ as a position of $s_l$ in $L_{l+h}$, and $S_{M-1}$ as $(s_1,\ldots,s_{l},s_{l-h+1},\ldots,s_{l-1})$ (in case of $h=0$ we put $S_{M-1}= (s_1,\ldots, s_l)$). 
\end{proof}
Let $T_M$ be the number of sequences $d_1,\ldots,d_M$ satisfying \itemref{item-dj<=1}, \itemref{item-sum-of-the-prefix>=1} and additionally $\sum_{j=1}^Md_j=1$. Such sequences are in close relation to plane trees, and are known to be enumerated by Catalan numbers, i.e., $T_{M+1}=C_M=\frac{1}{M+1}\binom{2M}{M}=o(4^M)$. Note that every feasible sequence of differences in a log has total sum less than $n$ (as Algorithm \ref{alg-choosing-from-lists} never terminates). The number of sequences satisfying \itemref{item-dj<=1}, \itemref{item-sum-of-the-prefix>=1} but with total sum equal $k$ (fixed $k\geq 1$) is at most $T_{M}$. Thus, we conclude that the number of all feasible difference sequences of size $M$ is at most $n\cdot T_{M}$. Clearly, for every feasible sequence of differences $D$ the number of sequences which can occur in log with $D$ is at most $4^n$. Since the number of logs is exactly $4^{M}$ we get
\[
 4^{M} \leq n \cdot T_M \cdot 4^n = o(4^M)
\]
which is a contadiction for large enough $M$. This means that the number of realizations which do not generate a  nonrepetitive sequence of length $n$ is smaller than the number of all realizations.
\end{proof}

\section{Preliminaries}\label{sec:preliminaries}
We make some use of generating functions theory. We consider only algebraic functions. A generating function $t(z)= \sum_n T_n z^n$ with positive radius of convergence is algebraic if there exists a nonconstant polynomial $P(z,t)\in \Cpl[z,t]$  (\emph{defining polynomial}) such that $P(z,t(z))$ is constantly zero within the disc of convergence of $t(z)$. It is a well known fact that, if the radius of convergence of $\sum_n T_n z^n$ is strictly greater than $\alpha$, then $T_n= o(\alpha^{-n})$. The following observation is fundamental in analysis of algebraic generating functions, the thorough study of which can be found in \cite{FS09} (chapter VII.7).
\begin{obs}
	Let $t(z)= \sum_n T_n z^n$ be a nonpolynomial algebraic generating function with defining polynomial $P(z,t)$. Then the radius of convergence of $t(z)$ is one of the roots of the discriminant of $P(z,t)$ with respect to the variable $t$ (i.e. the resultant of $P(z,t)$ and $\partial_t P(z,t)$ with respect to $t$).
\end{obs}

The coefficients of the functions we use are nonnegative integers. In such cases it is easy to see that the radius of convergence is not greater than 1, whenever a function has infinite number of nonzero coefficients.  In order to bound the growth of the sequence of coefficients of such a function,  we calculate the discriminant of its defining polynomial $P(z,t)$ with respect to the variable $t$, and look for its positive real root in the interval $(0,1]$. If there is only one such root, it must be the radius of convergence of the function.

\section{The erase-repetition game}\label{sec:erase-repetition-game}
\begin{theorem}\label{thm:erase-repetition-game}
In the erase-repetition game over a symbol set of size 8, there exists a strategy for Ann to build an arbitrarily long nonrepetitive sequence.
\end{theorem}
\begin{proof}
We fix $n$ and prove that Ann has a strategy to build a nonrepetitive sequence of size $n$. In fact, the strategy for Ann will be randomized and we will show that for every strategy of Ben there is an evaluation of random experiments leading to the sequence of size $n$ against that strategy. The fact that for every strategy of Ben there is a strategy for Ann to build a sequence of size $n$ implies that Ann simply has a strategy to build such a sequence.

Let $C$ be the size of a symbol set. The argument to be presented turns out to work for $C\geq 8$. The strategy for Ann is the following: choose a random element distinct from the last three symbols in the sequence constructed so far. In this setting, Ann does not generate repetitions of size $1$, $2$ and $3$. Obviously, Ben can cause many repetitions of size $1$ but repetitions of size $2$ and $3$ are not possible. Indeed, in order to get a repetition of the form 'abcabc' the last three symbols must be generated by Ben. Consider Ann's move just before Ben puts 'b' in the repeated block. As she could not play preceding symbol 'a' she must have invoked a repetition. But all her repetitions are of size at least $4$ and therefore the repeated block must have ended with 'abca'. This would mean that she played 'a' which is not possible as this symbol is not distinct from the last three in the current sequence at that step. Analogous argument proves that repetitions of size 2 are also impossible.

Fix $n$ and a strategy of Ben. Take  $M$ sufficiently large and consider possible scenarios of the first $2M$ moves of the game against that fixed Ben's strategy. Suppose, for a contradiction, that the size of a sequence  after $2M$ moves is always (for any evaluation of Ann's choices) less than $n$. Ann generates exactly $M$ elements. Let $r_j$ ($1\leq j\leq M$) be the $j$th symbol generated by Ann. Clearly, $r_1,\ldots,r_M$ is a sequence of random variables with at least $(C-3)^M$ possible evaluations. When we fix an evaluation of $r_1,\ldots,r_M$ the course of the whole game is determined.

Let $h_j$ ($1\leq j\leq 2M$) be the length of the sequence generated after $j$ moves (including possible erasure invoked by the $j$th move) and let $d_1,\ldots,d_{2M}$ be the sequence of differences: $d_1=1$, $d_j=h_j-h_{j-1}$ for $2\leq j\leq 2M$. Note that $d_j=1$ means that there is no erasure after $j$th move and $d_j<1$ indicates that repeated block of size $\norm{d_j}+1$ was removed. A pair $(D,S)$ is a \emph{game log} and $D$ is \emph{feasible} if there is an evaluation of $r_1,\ldots,r_M$ such that $D$ is the sequence of differences and $S$ is the final sequence produced after $2M$ moves. A pair $(D,S)$ is a \emph{reduced game log} if it is a log but with all zeros in $D$ erased. Note that any sequence of differences $D=(d_1,\ldots,d_m)$ in a reduced log satisfies:
\begin{enumeratei}
\item $m\leq 2M$,\label{item-erase-repetition-m<=2M}
\item $d_j\in\set{1,-3,-4,-5,\ldots}$, for all $1\leq j\leq m$,\label{item-erase-repetition-dj<=1}
\item $\sum_{j=1}^k d_j\geq 1$, for all $1\leq k\leq m$.\label{item-erase-repetition-sum-of-the-prefix>=1}
\end{enumeratei}

\begin{claim}
Every reduced log corresponds to a unique evaluation of $r_1,\ldots,r_M$.
\end{claim}
\begin{proof}
Given a reduced log $((d_1,\ldots,d_m),S_m)$ with $S_m=(s_1,\ldots,s_l)$, we decode all random choices taken by Ann in two steps. First we reconstruct the sequence $x_1,\ldots,x_m$ of all symbols introduced in the game except those (of Ben) generating repetitions of size $1$. The introduced symbols generating repetitions of size $1$ are called \emph{bad}, other symbols are \emph{good}. The move is good (bad) if a good (bad) symbol is introduced. The number of good moves played is exactly the size of the difference sequence in the reduced log, namely $m$. Note that $S_m$ is the sequence formed after the $m$th good move (even if a bad move is played afterwards, it does not change the sequence).

We reconstruct the sequence of good symbols backwards, i.e., we first decode $x_m$, which is the last good symbol introduced, and the sequence $S_{m-1}$ constructed after $m-1$ good moves. Then, by simple iteration, we extract all the remaining good symbols $x_{m-1},\ldots,x_1$.

If $d_m=1$, then the $m$th good symbol introduced did not invoke a repetition. Thus, the last good symbol introduced is the last symbol of the final sequence, i.e. $x_m=s_l$ and $S_{m-1}=(s_1,\ldots,s_{l-1})$.

If $d_m\leq 0$, then some symbols were erased after the $m$th good move. But since we know the size of the repetition, namely $h=\norm{d_m}+1$, and only one half of it was erased, we can read and copy the first part of the repeated block to restore $S_{m-1}=(s_1,\ldots,s_{l},s_{l-h+1},\ldots,s_{l-1})$ and $x_m=s_l$.

Once we get all $x_1,\ldots,x_m$, we read the sequence from the beginning and check whether the symbols agree with the strategy of Ben we fixed. The difference appears only where Ben introduces a bad symbol. There we extend the sequence with this symbol and continue. This way we reconstruct the sequence of all symbols introduced in the game and clearly every second symbol is chosen by Ann.
\end{proof}

By a \emph{game walk} we mean a sequence $d_1,\ldots,d_m$ satisfying \itemref{item-erase-repetition-dj<=1}, \itemref{item-erase-repetition-sum-of-the-prefix>=1} and additionally $\sum_{j=1}^m d_j=1$. Let $T_m$ be the number of gamewalks of length $m$. By our assumption that Ann never wins, every feasible sequence of differences in a reduced log sums up to a number smaller than $n$. The number of sequences of size $m$ satisfying \itemref{item-erase-repetition-dj<=1}, \itemref{item-erase-repetition-sum-of-the-prefix>=1} but with a total sum $k$ (for fixed $k\geqslant1$) is bounded by $T_{m+3}$ (just append two '1's and '$-(k+1)$' to the end). Note also that $T_m\leq T_{m+1}$ for $m>1$. Indeed, for a given sequence $d_1,\ldots,d_m$ let $i$ be the least index with $d_i<0$ (there must be such provided $m>1$). Then $d_1,\ldots,d_{i-1},1,d_i+1,d_{i+1},\ldots,d_m$ is a sequence counted by $T_{m+1}$ and this extension is injective. Finally, all feasible sequences of differences are of size at most $2M$. All this yields that the number of feasible difference sequences in a reduced log is at most $2M\cdot n\cdot T_{2M+3}$. For a given feasible sequence of differences $D$, the number of final sequences which can occur with $D$ in a reduced log is bounded by $C^n$. Thus, the number of reduced logs is bounded by
\[
2M\cdot n\cdot T_{2M+3}\cdot C^n.
\]

We turn to the approximation of $T_{2M}$. Every game walk $d_1,\ldots,d_m$ is either a single step up (i.e., $m=1$, $d_1=1$), or it can be uniquely decomposed into $\norm{d_m}+1$ subsequent game walks of total length $m-1$. The $j$th component of the decomposition is the substring between the last visit of height $j-1$ and the last visit of height $j$ (i.e. between the last $k$ such that $\sum_{i=1}^k d_i= j-1$ and last $l$ such that $\sum_{i=1}^l d_i= j$). This description together with the fact that if $m>1$, then $\norm{d_M}+1\geq4$, certify that the generating function $t(z)= \sum_{n\in N} T_n z^n$ satisfies the following functional equation:
\[
t(z)= z+z(t(z)^4+t(z)^5+\ldots) ,
\]
where the right hand side is $z+ z \frac{t(z)^4}{1-t(z)}.$ From this equation we extract a polynomial 
\[
P(z,t)= z t^4  +t^2 -(1+z) t + z
\]
 that defines $t(z)$.
In the standard way we calculate the discriminant polynomial obtaining:
\[
	-4	-19 z + 32 z^2 - 2 z^3 + 36 z^4 + 229 z^5 .
\]
This polynomial has only one positive real root equal to $\rho=0.457\ldots>{5}^{-\frac{1}{2}}$. Pick any $\alpha$ with $\rho^{-2}<\alpha<5$. Then $T_{2M} =o(\alpha^M)$.

By the claim the number of realizations is exactly the number of reduced logs. That gives
\[
(C-3)^M \leq 2M \cdot n \cdot T_{2M+3} \cdot C^n = 2M \cdot n \cdot o(\alpha^{M}) \cdot C^n = o(5^M).
\]
Thus for $C\geq 8$ and sufficiently large $M$ we obtain a contradiction.
\end{proof}

\section{The nonrepetitive game}\label{sec:nonrepetitive-game}
\begin{theorem}\label{thm:Pegdens-game}
In the nonrepetitive game over a symbol set of size 6, there is a strategy for Ann  to build an arbitrarily long sequence with no repetitions of size greater than $1$.
\end{theorem}
\begin{proof}
We fix $n$ and prove that Ann has a strategy to build a sequence of size $n$ without repetitions of size greater than $1$. As before we consider randomized Ann's strategy and we  show that for every strategy of Ben there is an evaluation of random experiments leading to the generation of a nonrepetitive sequence of size $n$. This means that Ben cannot have winning strategy. Therefore, there exists a winning strategy for Ann.

In this proof, by a repetition we mean a repetition of size greater than $1$.

Let $s_1,\ldots,s_{m-1}$ be the sequence already generated in the game and suppose that it is Ann's turn ($m$ is odd). The strategy for Ann goes as follows: choose any symbol at random, but
\begin{enumeratei}
 \item exclude $s_{m-2}$,\label{item: excluding repetitions of size 2}
 \item if $s_{m-1}=s_{m-4}$, then exclude $s_{m-3}$,\label{item: excluding repetitions of size 3}
 \item if only one symbol has been excluded in \ref{item: excluding repetitions of size 2} and \ref{item: excluding repetitions of size 3}, then exclude $s_{m-4}$.\label{item: excluding repetitions of size 4}
\end{enumeratei}

This stategy explicitly ensures that no repetitions of size $2$ and $3$ occur in the game. It turns out that also repetitions of size $4$ are avoided. Suppose for a contradiction that at some point in the game a sequence with a suffix of the form $x_1x_2x_3x_4x_1x_2x_3x_4$ is produced. Suppose also that Ann introduces the last symbol, namely $x_4$. As she did not prevent a repetition of size $4$, the rule \itemref{item: excluding repetitions of size 4} of the strategy did not exclude a symbol and therefore rule \itemref{item: excluding repetitions of size 3} must have been invoked. In particular, $x_3=x_4$. But this means that in the previous move of Ann (when she introduced $x_2$ in the repeated block) the symbols excluded by \itemref{item: excluding repetitions of size 2} and \itemref{item: excluding repetitions of size 3} were the same, so, rule \itemref{item: excluding repetitions of size 4} must have been applied. But that rule excludes $x_2$, a contradiction. Analogous reasoning works for the case when Ben finishes a repetition of size $4$.

Fix a strategy for Ben. We simulate the play between randomized Ann and this fixed strategy, and whenever a repetition of size $h$ occurs in the $m$th move (of the real game), we backtrack to the  move $m-h+1$. This means that
we remove the whole repeated segment and continue the simulation starting from the move $m-h+1$ again (with independent random experiments).

A \emph{search sequence} is the sequence of consecutive symbols chosen by players in the simulation. 
Note that it is not possible for Ben to introduce three symbols in a row in the search sequence. 
Indeed, if he introduces two symbols in a row, then there must have been a repetition (of odd size) after the first symbol. 
Thus, the second one is the same as the symbol just erased at this position (as Ben's strategy is fixed in the simulation). 
This means that the second symbol could not generate repetition and therefore Ann is next to play in the simulation.

The \emph{weight} of a search sequence is the number of symbols chosen by Ann in the sequence. Fix 
$M$ large enough. We are going to show that there is a scenario of the first $M$ random experiments (first $M$ moves of Ann)
leading the simulation to an outcome sequence of size $n$. This will prove that Ann has a strategy to build a sequence
of size $n$ against the fixed strategy of Ben. For a contradiction we suppose that all outcome sequences generated after $M$ moves of Ann in the simulation are of length less than $n$ for all possible evaluations of random experiments.

Clearly, a search sequence of weight $M$ is uniquely determined by the sequence of $M$ Ann's choices. Let $r_1,\ldots,r_M$
be the symbols chosen by Ann. As she always chooses  one symbol out of at least 
$C-2$ symbols, the sequence $r_1,\ldots,r_M$ has at least $(C-2)^M$ possible evaluations. A search sequence induced by an evaluation of $r_1,\ldots,r_M$
is called a \emph{realization} of this evaluation.

Let $h_j$ be the length of the current sequence just \emph{before} the $j$th step (move) of the simulation. The sequence $(h_j)$ is called a \emph{height sequence}. If Ann introduces a symbol in the $j$th step, then her next move is in step $k\in\{j+1,j+2,j+3\}$ (as Ben never plays three times in a row). There are only few possible extensions of the height sequence from $h_j$ up to $h_k$:
\begin{enumeratenum}\setcounter{enumi}{-1}
\item Ann makes no repetition in the $j$th step and Ben makes no repetition in the  $(j+1)$th step. In this case $k=j+2$ and $h_{j+1}=h_j+1$, $h_{j+2}=h_j+2$.\label{item:type0}
\item Ann makes a repetition of odd size, at least $5$, in the $j$th step and therefore she plays again in the $(j+1)$th step. Here $k=j+1$ and $h_{k}\leq h_j-4$.\label{item:type1}
\item Ann makes a repetition of even size, at least $6$, in the $j$th step and Ben plays no repetition in the $(j+1)$th step. Here $k=j+2$ and $h_k \leq h_j-4$.\label{item:type2} 
\item Ann makes no repetition in the $j$th step and Ben produces a repetition of even size, at least $6$, in the $(j+1)$th step. Here again $k=j+2$ and $h_k\leq h_j-4$.\label{item:type3}
\item Ann makes no repetition in the $j$th step. Ben makes a repetition of odd size, at least $5$, in the $(j+1)$th step. Then he plays no repetition in the $(j+2)$th step. Here $k=j+3$ and $h_k\leq h_j-2$.\label{item:type4}
\end{enumeratenum}
We want to get rid of some redundancy in the height sequence. More precisely, we encode the sequence of heights into its subsequence consisting of $h_j$'s corresponding to Ann's moves with a little extra information. Let $h_j,h_k$ be again the heights of the current sequence right before any two consecutive moves of Ann. Note that
\begin{itemize}
 \item If $h_k>h_{j}$, then the sequence of heights between $h_j$ and $h_k$ is of type \ref{item:type0}.
 \item If $h_k=h_{j}-2$, then the sequence of heights between $h_j$ and $h_k$ is of type \ref{item:type4}.
 \item If $h_k\leq h_j-4$, then the sequence of heights between $h_j$ and $h_k$ is of type \ref{item:type1}, \ref{item:type2},\ref{item:type3} or \ref{item:type4}.
\end{itemize}
Therefore, in order to record the whole height sequence it is enough to remember the subsequence $h'_1,\ldots,h'_M$ of heights corresponding to Ann's moves and additionally, if $h'_{j+1}\leq h'_{j}-4$, to record  $\type(h'_{j},h'_{j+1})\in\set{1,2,3,4}$, which is the type of the original height sequence between symbols corresponding to $h'_j$ and $h'_{j+1}$.

Finally, note that all the $h'_j$'s are even (as the current sequence before Ann's move contains an even number of symbols). The \emph{reduced sequence of differences} is: $d_1=1$, $d_{j+1}=(h'_{j+1}-h'_{j})/2$ for $1\leq j< M$, and the \emph{type function} $\type(d_{j+1})=\type(h_{j},h_{j+1})$, provided the latter is defined. Note that
\begin{enumeratei}
\item $d_j\leq 1$,\label{item-nonrepetitive-dj<=1}
\item $\sum_{j=1}^k d_j\geq 1$, for all $1\leq k\leq M$,\label{item-nonrepetitive-sum-of-the-prefix>=1}
\item $\type(d_j)$ is defined if and only if $d_j\leq -2$.\label{item-types}
\end{enumeratei}

A pair $((D,\type),S)$ is a \emph{search log} if there is an evaluation of $r_1,\ldots,r_M$ such that $D$ is the reduced sequence of differences in the realization of $r_1,\ldots,r_M$, $\type$ is a type function of $D$, and $S$ is the final sequence produced after $M$ steps of Ann in this realization of the search procedure.
\begin{claim}
Every search log corresponds to a unique evaluation of $r_1,\ldots,r_M$.
\end{claim}
\begin{proof}
Given a search log $((D,\type_D),S)$ where $S=(s_1,\ldots,s_l)$ we decode the evaluation of $r_1,\ldots,r_M$ in a few steps. First we extract the height sequence $h_1,\ldots,h_m$ from $(D,\type_D)$ and put additionally $h_{m+1}=\norm{S}$. Now, we are going to describe how to reconstruct the sequence $x_1,\ldots,x_m$ of all symbols introduced in the simulation. This is done in backward direction, i.e., we decode first $x_m$ and the sequence $S_{m-1}$ constructed after $m-1$ steps of the simulation. Then by simple iteration we extract all the remaining symbols $x_{m-1},\ldots,x_1$.

If $h_{m+1}-h_m=1$, then the introduction of $x_m$ did not invoke a repetition. Thus, $x_m$ is the last symbol in the final sequence $S$, i.e., $x_m=s_l$ and $S_{m-1}=(s_1,\ldots,s_{l-1})$.

If $h_{m+1}-h_m \leq 0$, then some symbols were erased after the introduction of $x_m$. But we know the size of the repetition, namely $h=\norm{h_{m+1}-h_m}+1$, and since only one half of it was erased, we can copy the appropriate block to restore  $s_{m-1}=(s_1,\ldots,s_{l},s_{l-h+1},\ldots,s_{l-1})$ and $x_m=s_l$.

Once we get all $x_1,\ldots,x_m$ we read the sequence from the beginning and track the current sequence in the simulation. Every time the current sequence is of even length the next symbol is introduced by Ann.
\end{proof}

By a \emph{typed search walk} we mean a pair $((d_1,\ldots,d_M),\type)$ satisfying \ref{item-nonrepetitive-dj<=1}, \ref{item-nonrepetitive-sum-of-the-prefix>=1}, \ref{item-types} and additionally $\sum_{j=1}^M d_j=1$. Let $T_M$ be the number of typed search walks of length $M$ (i.e., $D$ is of length $M$). By our assumption that Ann never wins, every feasible sequence of differences in a typed search walk sums up to less than $n$.  The number of typed search walks of length $M$ satisfying \ref{item-nonrepetitive-dj<=1}, \ref{item-nonrepetitive-sum-of-the-prefix>=1}, \ref{item-types} with total sum  $k$ (fixed $k\geqslant1$) is at most $T_{M+1}$ (just append $-(k-1)$ to the end and pick arbitrary type, if necessary). Furthermore, $T_m\leq T_{m+1}$ for $m>1$. All this implies that the number of feasible typed search walks is $n\cdot T_{M+1}$. For a given feasible typed search walk $(D,\type)$ the number of final sequences which can occur with $(D,\type)$ in a search log is bounded by $C^n$. Thus, the number of reduced logs is bounded by
\[
n\cdot T_{M+1}\cdot C^n.
\]

We turn to the approximation of $T_{m}$. Every typed search walk $((d_1,\ldots,d_m),\type)$ is either a single step up (i.e., $m=1$, $d_1=1$), or it can be uniquely decomposed into $\norm{d_m}+1$ subsequent search walks of total length $m-1$ and additionally into the type of $d_m$ if it is defined (i.e., if $d_m\leq -2$). This decomposition (analogous as in the proof of Theorem \ref{thm:Pegdens-game}) gives the following functional equation for the generating function $t(z)$:
\[
	t(z)=  z + zt^2(z) + 4z(t^3(z)+t^4(z)+t^5(z)+\ldots),
\]
where $z$ stands for a trivial one-step-up walk, $zt^2(z)$ stands for the case $d_m=-1$ in which $d_m$ has no type, and the last term stands for the case $d_m\leq -2$. The right hand side of the equation is in fact equal to $z+ zt^2(z) + 4 z \frac{t^3(z)}{1-t(z)}$. From that form we derive the defining polynomial for $t(z)$:
\[
	P(z,t)= -t + t^2 + z - t z + t^2 z + 3 t^3 z.
\]
In the standard way we calculate the discriminant polynomial obtaining:
\[
	1 +12 z - 24 z^2 - 80 z^3 - 288 z^4.
\]
The radius of convergence of $t(z)$ is one of the roots of the above polynomial. This polynomial has only one positive real root in $\rho=0.2537..>4^{-1}$. Therefore $T_M= o(4^M)$.

By the claim, the number of realizations is exactly the number of search logs. That gives
\[
(C-2)^M \leq n\cdot T_{M+1}\cdot C^n=o(4^M).
\]
Therefore for $C\geq6$ and sufficiently large $M$ we obtain a contradiction.
\end{proof}

\section{Final remarks}
The expected running time of the algorithm is linear in $n$ for lists  of size at least 4. It is immediate for lists of size 5, and needs a little effort for size 4. The computational experiments suggests different behaviour for size 3. This somehow explains the difficulty of Conjecture \ref{conj}. It might be also the case that the list version of Thue's theorem does not hold, as it goes with the list version of the Four Color Theorem, although every planar graph is colorable from lists of size 5 \cite{Tho94}.

It is natural to try a similar approach for other Thue-type problems,
especially for those in which the Lov\'{a}sz local lemma has been previously
successfully applied. One such topic concerns graph-theoretic analogues of
nonrepetitive sequences. A coloring of the vertices of a graph $G$ is \emph{%
nonrepetitive} if sequences of colors on all simple paths of $G$ are
nonrepetitive. The minimum number of colors needed is denoted by $\pi (G)$.
This parameter is bounded for graphs with bounded degree \cite%
{AGHR02}, as well as for graphs with bounded treewidth \cite{BV07}, \cite{KP08}. A major challenge of this area is to settle whether $\pi
(G)$ is bounded by a constant for all planar $G$.

The ideas behind the erase-repetition algorithm already led to the proof \cite{KM} that for every tree and lists of size $4$ one can choose a coloring with no \emph{three} consecutive identical blocks on any simple path. This fits to the recent construction from \cite{FOOZ11} proving that no constant-size of lists guarantees a nonrepetitive coloring of a tree chosen from these lists.

Another direction is to look for stronger versions of nonrepetitive
sequences. Here is an interesting variation due to Erd\H{o}s \cite%
{Erd61}. A sequence $S$ is \emph{strongly nonrepetitive} if no two
adjacent blocks of $S$ are permutations one of another. It is known that
there are arbitrarily long strongly nonrepetitive sequences over four
symbols \cite{Ker92}. But is it true that one can choose strongly
nonrepetitive sequences from any collection of lists of sufficiently large
size?

\bibliographystyle{siam}
\bibliography{zigzags}
\end{document}